\documentclass[12pt,reqno]{amsart}
\usepackage{amsfonts,amssymb,amsmath}
\usepackage{datetime}
\usepackage{xcolor}
\usepackage{enumerate}
\usepackage{multirow}
\usepackage{comment}

\definecolor{fgreen}{RGB}{44,144, 14}

\renewenvironment{proof}{{\bfseries Proof.}}{\qed}

\numberwithin{equation}{section} 
\newtheorem{theorem}{Theorem}[section]

\newtheorem{lemma}[theorem]{Lemma} 
\everymath{\displaystyle}
\theoremstyle{definition}
\newtheorem{definition}[theorem]{Definition}

\usepackage[colorlinks=true,citecolor=cyan,  urlcolor=blue,linkcolor=blue]{hyperref}

\def\R{\mathbb R}

\def\C{\mathbb C}

\def\R{\mathbb R}

\def\R{\mathbb {R}}
\def\C{\mathbb {C}}

\begin{document}

	\title[Equi-Baire One Families]{Equi-Baire One Families of Möbius Transformations and One-Parameter Subgroups of $\mathrm{PSL}(2,\mathbb{C}$)}
    	\author[S.  Dutta, Vanlalruatkimi and Jonathan Ramdikpuia]{Sandipan Dutta, Vanlalruatkimi and Jonathan Ramdikpuia}

	\address{Department of Mathematics and Computer Science, Mizoram University, Aizawl 796004, Mizoram}
	\email{sandipandutta98@gmail.com}
    \address{Department of Mathematics and Computer Science, Mizoram University, Aizawl 796004, Mizoram}
	\email{aruatihmar39@gmail.com }
    \address{Department of Mathematics and Computer Science, Mizoram University, Aizawl 796004, Mizoram}
	\email{jonathanramdikpuia@gmail.com }

	
	\date{ @\currenttime , \today}
    \begin{abstract}
    We study the Equi-Baire one property families of M\"obius transformations on the Riemann sphere. For a loxodromic map $f$, we show its iterates $\{f^n\}$ form an orbitally Equi-Baire one family on the attracting basin. For a one-parameter subgroup $\{f_t \}$, we prove it is Equi-Baire one on all compact sets of $\widehat{\mathbb{C}}$ if and only if the subgroup is relatively compact in $\mathrm{SL}(2,\mathbb{C})$. This provides a dynamical characterization of the Equi-Baire one condition for M\"obius families.
\end{abstract}
	\maketitle 
	
	\section{Introduction} 
The theory of Möbius transformations and the action of the group $\mathrm{PSL}(2,\mathbb{C})$ on the extended complex plane $\widehat{\mathbb{C}} = \mathbb{C} \cup \{\infty\}$ has been a central topic in complex analysis and geometry for more than a century. Each element of $\mathrm{SL}(2,\mathbb{C})$ acts on $\widehat{\mathbb{C}}$ by a fractional linear transformation
$$
f(z) = \frac{az + b}{cz + d}, \quad a,b,c,d \in \mathbb{C}, \ ad - bc = 1,$$
and this action preserves the conformal structure of the Riemann sphere. The study of these transformations  was later developed systematically in Beardon~\cite{BA}, Maskit~\cite{MB}, and Goldman~\cite{GW}. Through the classification of elements into elliptic, parabolic, hyperbolic, and loxodromic types based on the trace of their representing matrices, one gains a deep geometric understanding of discrete subgroups of $\mathrm{PSL}(2,\mathbb{C})$ and their dynamics on the Riemann sphere. The loxodromic case, in particular, exhibits rich dynamical behavior, possessing two distinct fixed points, one attracting and one repelling which serves as a prototype for studying contraction and expansion phenomena in complex dynamical systems~\cite{RAT}.

In parallel with these geometric and dynamical developments, the notion of Baire one functions has become an important tool in modern topology and analysis. A function $f : X \to Y$ between metric spaces is said to be of \emph{Baire class one} if it can be expressed as the pointwise limit of a sequence of continuous functions. The concept generalizes continuity and provides a framework for describing functions that are ``almost continuous’’ in a precise topological sense. Foundational studies of such functions appear in the works of Bourgain, Fremlin, and Talagrand~\cite{BOU}, Fenecios et al.~\cite{FJ}, and Horowitz and Todorcevic~\cite{HOR}. More recently, Alikhani-Koopaei~\cite{AK,ALI, AK19} introduced the idea of \emph{Equi–Baire one families}—a generalization of equicontinuity—where the uniformity of continuity is extended from individual continuous maps to families of Baire one functions. This concept allows a unified treatment of uniform convergence and continuity phenomena in settings where classical equicontinuity fails.

Motivated by this interplay between dynamical systems and functional analysis, the present paper investigates the \emph{Equi–Baire one properties} of families of Möbius transformations acting on $\widehat{\mathbb{C}}$. Our focus lies in two fundamental settings: first, the discrete family of iterates $\mathcal{F} = \{f^n : n \ge 0\}$ associated with a loxodromic transformation $f \in \mathrm{SL}(2,\mathbb{C})$; and second, the continuous one–parameter subgroup $\mathcal{F} = \{f_t = \exp(tA) : t \in [0,\infty)\}$ determined by a matrix $A \in \mathrm{SL}(2,\mathbb{C})$. In the first case, we show that the iterates of a loxodromic map converge uniformly to its attracting fixed point on compact subsets of its attracting basin, implying that the sequence $\{f^n\}$ is orbitally Equi–Baire one at every point in this region. In the second case, we classify one–parameter subgroups of $\mathrm{SL}(2,\mathbb{C})$ and establish precise conditions under which the corresponding family $\{f_t\}$ is Equi–Baire one on compact subsets of  $\widehat{\mathbb{C}}:=\C\cup\{\infty\}$.

The main results demonstrate that the family $\{f^n\}$ generated by a loxodromic Möbius transformation forms an Equi–Baire one family on compact neighborhoods of points in its attracting basin, with an explicit $\delta$–function constructed from the chordal metric witnessing this property. Moreover, for a one–parameter subgroup $\{f_t = \exp(tA)\}$, the family $\{f_t\}$ is Equi–Baire one on every compact subset of $\widehat{\mathbb{C}}$ if and only if it is relatively compact in $\mathrm{SL}(2,\mathbb{C})$, i.e., conjugate to a subgroup of $\mathrm{SU}(2)$. These results connect analytic regularity in the Baire sense with geometric dynamics on the Riemann sphere, illustrating how the analytic behavior of families of maps reflects the underlying geometric structure of the transformation group.

\begin{theorem}
\label{thm main1}
Let $f \in \mathrm{SL}(2,\mathbb{C})$ act on $\widehat{\mathbb{C}}$ by the corresponding Möbius map. 
Suppose $f$ is loxodromic, i.e.\ conjugate to $z \mapsto \lambda z$ with $|\lambda| \ne 1$. 
Then for any point $x$ in the attracting basin of the attracting fixed point $p$, the family of iterates 
$$
F = \{f^n : n \ge 0\}
$$ 
is orbitally Equi–Baire one at $x$. 
Moreover, an explicit $\delta$–function witnessing this property is given by
$$
S(x,r) = \sup_{n \ge 0}\, \sup_{y \in B_r(x)} d(f^n(y), f^n(x)),
\qquad S(x,r) \to 0 \text{ as } r \to 0.
$$
\end{theorem}
\begin{theorem}
    \label{thm main2}
Let $\{f_t\}_{t \in [0,\infty)}$ be a one–parameter subgroup of $\mathrm{SL}(2,\mathbb{C})$, where 
$f_t(A) = \exp(tA)$ acts on $\widehat{\mathbb{C}}$.  
Then the family $F = \{ f_t : t \in [0,\infty) \}$ is Equi–Baire one on every compact 
subset of $\widehat{\mathbb{C}}$ if and only if the subgroup $\{\exp(tA) : t \in \mathbb{R}\}$ is relatively compact in \(\mathrm{SL}(2,\mathbb{C})\); equivalently, if and only if it is conjugate to a subgroup of $\mathrm{SU}(2)$.
\end{theorem}

\subsection{Organization of the paper.}
The paper is structured as follows. Section~\ref{sec2} recalls the basic notions and tools from the theory of Möbius transformations, Baire one functions, and the geometric structure of the Riemann sphere. Section~\ref{sec3} establishes several lemmas that will be used in the proofs of the main results. In Section~\ref{sec4}, we prove Theorem~\ref{thm main1} and Theorem~\ref{thm main2}, describing the Equi--Baire one behaviour of iterated Möbius transformations and one--parameter subgroups of $\mathrm{PSL}(2,\mathbb{C})$.

\section{Preliminaries}
\label{sec2}
In this section, we recall some fundamental notions and notations which will be used
throughout the paper.
\subsection{Möbius transformations and the group $\mathrm{SL}(2,\mathbb{C})$.}
Every element $f \in \mathrm{SL}(2,\mathbb{C})$ acts on the extended complex plane 
$\widehat{\mathbb{C}} = \mathbb{C} \cup \{\infty\}$ by the corresponding Möbius transformation
$$
f(z) = \frac{az+b}{cz+d}, \qquad ad - bc =1.
$$
The theory of Möbius transformations and their geometric properties is classical 
(see Beardon \cite{BA}, Maskit \cite{MB}).  
The projectivization $\mathrm{PSL}(2,\mathbb{C}) = \mathrm{SL}(2,\mathbb{C})/\{\pm I\}$ forms the full group of orientation–preserving
conformal automorphisms of the Riemann sphere $\widehat{\mathbb{C}}$ \cite{BA},\cite{GW}.

\subsection{Classification of elements of $\mathrm{SL}(2,\mathbb{C})$.}
Let $A=\begin{pmatrix} a & b \\ c & d \end{pmatrix}\in \mathrm{SL}(2,\mathbb{C})$ represent $f$.
The trace determines the type:
\begin{itemize}
\item $|{\rm tr}(A)| < 2 \text{ (elliptic)},$ 
\item $ |{\rm tr}(A)| = 2 \text{ (parabolic)}$, 
\item ${\rm tr}(A)\in\mathbb{R}, |{\rm tr}(A)|>2 \text{ (hyperbolic)},$ 
\item ${\rm tr}(A)\notin\mathbb{R} \text{ (loxodromic)}.$
\end{itemize}
One can see \cite{MB} and \cite{RAT} for proofs and dynamical interpretations. \\
 Background on the Lie group structure of $\mathrm{SL}(2,\mathbb{C})$, the exponential map, and the behavior of one–parameter subgroups generated by matrices in $\mathrm{SL}(2,\mathbb{C})$ can be found in 
 
 \cite{HEL} and \cite{KN}.

\begin{definition}[The chordal metric] 

The natural metric on $\widehat{\mathbb{C}}$ is the spherical (chordal) metric
$$
d(z_1,z_2)=\frac{|z_1-z_2|}
{\sqrt{(1+|z_1|^2)(1+|z_2|^2)}},
$$
extended by $d(z,\infty)=(1+|z|^2)^{-1/2}$.
\end{definition}
This metric is classical and widely used in the study of Möbius transformations 
\cite{BA}.  
With this metric, $\widehat{\mathbb{C}}$ becomes a compact metric space.
\begin{definition}[Attracting Basin]
Let $f: \widehat{\mathbb{C}} \to \widehat{\mathbb{C}}$ be a Möbius transformation, and suppose $p$ is an attracting fixed point of $f$ that is, $$f(p)=p, \qquad |f'(p)|< 1$$ Then the attracting basin of $p$, also called the basin of attraction which is defined as the set of all points $x$ whose forward iterates under $f$ which converge to $p$ \cite{CL}. Explicitly $$\mathcal{A}(p)=\{x \in \widehat{\mathbb{C}} :f^n (x)\xrightarrow{n\rightarrow\infty} p \}.$$
\end{definition}
\subsection{Baire one functions and Equi–Baire one families.}
\begin{definition}[Baire-one]
A function $f:X\to Y$ is of {\it Baire class one} if it is the pointwise limit of a sequence of 
continuous functions.     
\end{definition}
 Background references include Bourgain et al. \cite{BOU}, 
Fenecios et al. \cite{FJ}, and Horowitz et al. \cite{HOR}.  
The notion of Equi–Baire one families was introduced and developed by 
Alikhani–Koopaei \cite{ALI},\cite{AK}, extending classical notions of equicontinuity.
\begin{definition}[Equi-Baire one, \cite{ALI}]\label{equi-baire_def}
    Let $X,Y$ be two metric spaces then a family of functions $\mathcal{F}=\{f_t:X\rightarrow Y:\; t\in I\}$ is called and Equi-Baire one family if there exists one single sequence of continuous functions $g_1,g_2,\ldots$ maps from $X$ to $Y$ such that for every $f_i\in \mathcal{F} $ and every $x\in X$ we have $$g_n(x)\xrightarrow{n\rightarrow\infty}f_i(x).$$
\end{definition}
The $I$ in Definition \ref{equi-baire_def} denotes any indexing set.
\par Although we shall use the above definition, the classical $\epsilon-\delta$ definition will help the reader to see the difference between equicontinuity and equi-Baire.
\begin{definition}
     Let $(X,\rho),(Y,d)$ be two metric spaces and $\epsilon>0$. A family of functions $\mathcal{F}=\{f_t:X\rightarrow Y:\; t\in I\}$ is called an Equi-Baire one family if there exists a function $\delta:X\rightarrow \R^+$ such that, for all $x,y\in X$ satisfying $\rho(x,y)<\min\{\delta(x),\delta(y)\}$ we have 
     $$ d(f(x),f(y))\leq \epsilon,\text{ for all }f\in\mathcal{F}.$$
\end{definition}
We will frequently use the following result of Alikhani–Koopaei \cite[Th.~3.9]{ALI}.

\begin{lemma}\label{3.9}
    If $\{f_n\}$ is a uniformly convergent sequence of Baire one maps on a compact metric space,
then $\{f_n\}$ is an Equi–Baire one family.
\end{lemma}

\subsection{The Riemann sphere and stereographic projection.}
The extended complex plane $\widehat{\mathbb{C}}$ is identified with the sphere $S^2$ 
via stereographic projection \cite{BA}.  
The standard formulae for $\sigma$ and $\sigma^{-1}$ endow $\widehat{\mathbb{C}}$ with the usual 
conformal structure making Möbius transformations conformal automorphisms of $S^2$.

The action of $\mathrm{PSL}(2,\mathbb{C})$ on $\widehat{\mathbb{C}}$ is faithful and $3$-transitive 
\cite{BA}, \cite{GW}.  
These properties play a central role in the conjugacy classification of Möbius transformations.

\subsubsection{Conjugacy and fixed points.}
Two elements of $\mathrm{PSL}(2,\mathbb{C})$ are conjugate precisely when they are dynamically
equivalent (their actions differ only by a change of coordinate) \cite{BA}, \cite{MB}.  
The computation of fixed points via 
$$
cz^2+(d-a)z-b=0
$$
is given in \cite{BA}. In the loxodromic case, one fixed point is attracting and the other repelling 
\cite{RAT}, \cite{GW}.

\section{Useful Lemmas}
\label{sec3}
In this section, we establish some lemmas which will be used to prove our main theorems. 

\begin{lemma}
\label{lem1}
If $f \in \mathrm{SL}(2,\mathbb{C})$ is loxodromic, then there exists a Möbius transformation 
$h \in \mathrm{PSL}(2,\mathbb{C})$ and a complex number $\lambda$ with $0 < |\lambda| < 1$ such that 
$$
f = h^{-1} \circ g \circ h, \qquad g(z)=\lambda z.
$$
In particular, $f$ has two distinct fixed points $p,q \in \widehat{\mathbb{C}}$, where 
$p=h^{-1}(0)$ is attracting and $q=h^{-1}(\infty)$ is repelling.
\end{lemma}




\begin{lemma}
    \label{lem2}
Let $x$ lie in the attracting basin of the attracting fixed point $p$ of a loxodromic 
map $f \in \mathrm{SL}(2,\mathbb{C})$. 
Let $h \in \mathrm{PSL}(2,\mathbb{C})$ be a Möbius transformation such that 
$f = h^{-1} \circ g \circ h$, where $g(z) = \lambda z$ and $0<|\lambda|<1$. 
Then there exists a compact neighborhood $K$ of $x$ such that 
$$
f^n|_K = h^{-1} \circ g^n \circ h|_K \longrightarrow p
$$
uniformly on $K$ as $n \to \infty$.
\end{lemma}
\begin{proof}
Choose a small neighborhood $U \subset \widehat{\mathbb{C}}$ of $x$ such that $h$ is 
a biholomorphism from $U$ onto an open set $V = h(U) \subset \widehat{\mathbb{C}}$. 
Let $K = \overline{U}$ and $\overline{V} = h(K)$. 
Since $g(z) = \lambda z$ with $0<|\lambda|<1$, for any $z \in \overline{V}$ we have 
$$
|g^n(z)| = |\lambda|^n |z| \to 0 \quad \text{uniformly on } \overline{V}.
$$
Therefore $g^n(z) \to 0$ uniformly on $\overline{V}$. 
By continuity of $h^{-1}$ on the compact set $\overline{V}$, the convergence is 
preserved under $h^{-1}$, and hence 
$$
f^n(z) = h^{-1}(g^n(h(z))) \longrightarrow h^{-1}(0) = p
$$
uniformly on $K$. 
This shows that the iterates $\{f^n\}$ converge uniformly to the constant map $p$ 
on a compact neighborhood of $x$. 
\end{proof}

\begin{lemma}
   \label{lem3}
Let $K$ be the compact neighborhood of $x$ obtained in Lemma \ref{lem2}. 
Then each iterate $f^n|_K$ is a continuous (hence Baire one) function, and the 
sequence $\{f^n|_K\}_{n\ge0}$ converges uniformly on $K$ to the constant map 
$p$. Consequently, the family $\{f^n|_K\}_{n\geq0}$ is 
\textit{Equi–Baire one} on $K$.
\end{lemma}
\begin{proof}
From Lemma \ref{lem2} we have uniform convergence $f^n|_K \to p$ on the compact metric 
space $K$. Since each $f^n$ is a Möbius transformation, it is continuous on 
$\widehat{\mathbb{C}}$, and therefore Baire one. 

By Lemma \ref{3.9} (which states that a uniformly convergent sequence of Baire one 
functions on a compact metric space forms an Equi–Baire one family), it follows 
that the sequence $\{f^n|_K\}$ is Equi–Baire one on $K$. 

Hence, for every $\varepsilon>0$, there exists a $\delta>0$ (depending on $x$) 
such that for all $y$ satisfying $d(y,x)<\delta$,
$$
d(f^n(y), f^n(x)) < \varepsilon \quad \text{for all } n\geq 0.
$$
This proves that $\{f^n\}$ is orbitally Equi–Baire one at $x$.
\end{proof}

\begin{lemma}
    \label{lem4}
Let $d$ denote the chordal metric on $\widehat{\mathbb{C}}$. 
For $r>0$, define 
$$
S(x,r) = \sup_{n\ge0} \, \sup_{y \in B_r(x)} d(f^n(y), f^n(x)).
$$
Then:

\begin{enumerate}[(i)]
    \item $S(x,r)$ is finite for each fixed $r>0$;
    \item $S(x,r) \to 0$ as $r \to 0$;
    \item The function $r \mapsto S(x,r)$ is upper semicontinuous.
\end{enumerate}
\end{lemma}
\begin{proof}
For each fixed $r>0$, the set $\overline{B_r(x)} \cap K$ is compact and each 
$f^n$ is continuous on $K$. Hence $S(x,r)$ is finite as a supremum of finitely 
bounded quantities.

To prove the limit property, let $\varepsilon>0$. 
From Lemma~\ref{lem2}, $f^n \to p$ uniformly on $K$. 
Therefore, there exists $N \in \mathbb{N}$ such that 
$$
\sup_{y\in B_r(x)} d(f^n(y),p) < \frac{\varepsilon}{3} \qquad \text{for all } n \ge N.
$$
For the finitely many iterates $f^0, f^1, \dots, f^{N-1}$, continuity at $x$ 
provides $\delta>0$ such that 
$$
\sup_{y\in B_\delta(x)} d(f^n(y), f^n(x)) < \frac{\varepsilon}{3}, 
\qquad \text{for all } 0 \le n < N.
$$
Combining these estimates, we obtain for all $n \ge 0$ and 
$y \in B_\delta(x)$,
$$
d(f^n(y), f^n(x)) < \varepsilon.
$$
Hence $S(x,r) \to 0$ as $r \to 0$.

Finally, upper semicontinuity of $S(x,r)$ follows directly from the definition 
of the supremum, since enlarging $r$ can only increase (or preserve) the set 
over which the supremum is taken.
\end{proof}
\begin{lemma}
\label{lem5}
With the notation of the previous lemmas, there exists a constant $C>0$ such that 
for all $y$ sufficiently close to $x$ and for all $n \ge 0$,
$$
|f^n(y) - f^n(x)| \le C\,|\lambda|^n\,|h(y) - h(x)|.
$$
Consequently, one may take the simpler bound
$$
S(x,r) \le C'\,r,
$$
which also satisfies $S(x,r) \to 0$ as $r \to 0$.
\end{lemma}
\begin{proof}
Since $f = h^{-1} \circ g \circ h$ with $g(z) = \lambda z$, we have for any $y$
$$
f^n(y) = h^{-1}(\lambda^n h(y)), \qquad 
f^n(x) = h^{-1}(\lambda^n h(x)).
$$
By the mean–value inequality and the continuity of $h^{-1}$ on the compact set 
$V = h(K)$, there exists a constant $C>0$ such that 
$$
|h^{-1}(z_1) - h^{-1}(z_2)| \le C\,|z_1 - z_2|, 
\qquad \text{for all } z_1,z_2 \in V.
$$
Therefore,
$$
|f^n(y) - f^n(x)| 
= |h^{-1}(\lambda^n h(y)) - h^{-1}(\lambda^n h(x))|
\le C\,|\lambda|^n\,|h(y) - h(x)|.
$$
Since $h$ is locally Lipschitz near $x$, say 
$|h(y) - h(x)| \le L\,|y - x|$ for some $L>0$, we obtain
$$
|f^n(y) - f^n(x)| \le (CL)\,|\lambda|^n\,|y - x|.
$$
Taking the supremum over $y \in B_r(x)$ yields 
$S(x,r) \leq C' r$ for some $C' > 0$, which tends to $0$ as $r \to 0$. 
This provides an explicit $\delta$–function witnessing the Equi–Baire one 
property at $x$.
\end{proof}

\begin{lemma}\cite{MB}
\label{lem6}
Every one–parameter subgroup $\{ \exp(tA) : t \in \mathbb{R} \}$ of $\mathrm{SL}(2,\mathbb{C})$ 
is, up to conjugacy, one of the following four types:
\begin{enumerate}[(i)]
    \item \textbf{Elliptic type:} conjugate to a subgroup of $\mathrm{SU}(2)$.
    The matrix $A$ has purely imaginary eigenvalues $\pm i\theta$, $\theta \in \mathbb{R}$.
    The corresponding Möbius transformations act as rotations on $\widehat{\mathbb{C}}$.

    \item \textbf{Hyperbolic type:} conjugate to 
    $$  \left\{ 
        \begin{pmatrix} e^{\lambda t} & 0 \\ 0 & e^{-\lambda t} \end{pmatrix} 
        : t \in \mathbb{R} 
        \right\}, \qquad \lambda \in \mathbb{R}\setminus\{0\}.$$
    The action has two fixed points on $\widehat{\mathbb{C}}$, one attracting and one repelling.

    \item \textbf{Parabolic type:} conjugate to 
    $$   \left\{ 
        \begin{pmatrix} 1 & t \\ 0 & 1 \end{pmatrix} 
        : t \in \mathbb{R}
        \right\}.$$
    The corresponding Möbius transformations have a single fixed point on $\widehat{\mathbb{C}}$.

    \item \textbf{Loxodromic type:} diagonalizable over $\mathbb{C}$ with eigenvalues 
    $e^{(\alpha + i\beta)t}$ and $e^{-(\alpha + i\beta)t}$ where 
    $\alpha, \beta \in \mathbb{R}$, $\alpha \neq 0$ and $\beta \neq 0$.
    The action has two fixed points, one attracting and one repelling, combined with rotation.
\end{enumerate}
\end{lemma}





\begin{lemma}
\label{lem7}
Each $f_t:\widehat{\mathbb{C}}\to\widehat{\mathbb{C}}$ in the one-parameter family 
$\{f_t\}_{t\in[0,\infty)}$, $f_t=\exp(tA)$ with $A\in\mathrm{SL}(2,\mathbb{C})$, 
is a homeomorphism of $\widehat{\mathbb{C}}$. Consequently, for any compact 
$K\subset\widehat{\mathbb{C}}$ each $f_t|_K$ is continuous (hence Baire–one). 
However, the continuity of each $f_t$ does \emph{not} imply that the family 
$\mathcal{F}=\{f_t:t\in[0,\infty)\}$ is Equi–Baire one on $K$.
\end{lemma}

\begin{proof}
A Möbius transformation associated to a matrix $M\in \mathrm{SL}(2,\mathbb C)$ acts on the Riemann sphere
$\widehat{\mathbb C}$ by
$$
z\mapsto M\cdot z=\frac{az+b}{cz+d},\qquad M=\begin{pmatrix}a&b\\ c&d\end{pmatrix},$$
with the usual interpretation at the point $\infty$. Such maps are rational functions
whose denominators do not vanish identically; therefore they are continuous on 
$\widehat{\mathbb C}$. Since $\det M=1\neq 0$, every Möbius map is invertible and its inverse
is again a Möbius map (corresponding to $M^{-1}\in \mathrm{SL}(2,\mathbb C)$). Thus each $f_t=\exp(tA)$
is a homeomorphism of $\widehat{\mathbb C}$.

Restriction to a compact $K\subset\widehat{\mathbb C}$ preserves continuity, so every
$f_t|_K$ is continuous and in particular a Baire–one function on $K$.

The final assertion is a remark about the strength of the Equi–Baire one property:
being Baire–one individually is a pointwise approximation property for each map, while
being Equi–Baire one requires the same single sequence of continuous functions
$\{g_n\}$ on $K$ to converge pointwise to \emph{every} $f_t$ simultaneously.
For uncountable families (as here) this is a much stronger requirement and need not
hold in general — in later lemmas we produce dynamical obstructions (e.g.\ attracting
fixed points and collapsing behavior) that prevent a single approximating sequence
from working for the whole family.
\end{proof}

\begin{lemma}
\label{lem8}
Let $\{f_t\}_{t\ge0}$ be a one-parameter subgroup of $\mathrm{SL}(2,\mathbb C)$ acting on $\widehat{\mathbb C}$.
Let $K\subset\widehat{\mathbb C}$ be compact and suppose there exists a nonempty open set
$U\subset K$ and a point $p\in\widehat{\mathbb C}$ such that for some sequence $t_n\to\infty$ one has
$$ \lim_{n\to\infty} f_{t_n}(z)=p \qquad\text{for every } z\in U.$$
Then $\mathcal F=\{f_t:t\ge0\}$ is \emph{not} Equi–Baire one on $K$.
\end{lemma}

\begin{proof}
Assume, towards a contradiction, that $\mathcal {F}$ is Equi–Baire one on $K$.
Then there exists a single sequence of continuous maps $\{g_m\}_{m \geq 1}$ on $K$ such that
for every $t \geq 0$ and every $x \in K$ we have
$$ \lim_{m \to \infty} g_m(x)=f_t(x).$$

Pick any point $z_0\in U$. Because $f_{t_n}(z_0) \to p$ as $n \to \infty$, the values $f_{t_n}(z_0)$
eventually lie arbitrarily close to $p$. On the other hand, for each fixed index $n$ the assumption
that $g_m\to f_{t_n}$ pointwise implies
$$\lim_{m\to\infty} g_m(z_0)=f_{t_n}(z_0).$$
But the left-hand side is independent of $n$ (it is the limit of the single sequence $\{g_m(z_0)\}_m$),
so we must have
$$f_{t_n}(z_0)=\lim_{m\to\infty} g_m(z_0) \quad\text{for all } n.$$
Taking $n\to\infty$ gives $p=\lim_{n\to\infty} f_{t_n}(z_0)=\lim_{m\to\infty} g_m(z_0)$.

Now evaluate the same limit for a fixed finite time, for example $t=0$. Since $g_m\to f_0$ pointwise,
we have
$$\lim_{m\to\infty} g_m(z_0)=f_0(z_0)=z_0.$$
Combining the two expressions for $\lim_{m\to\infty} g_m(z_0)$ yields $p=z_0$, contradicting the
assumption that $f_{t_n}(z_0)\to p$ with $p$ different from $z_0$ (which occurs whenever at least
one $f_{t}$ is not the identity on $U$ — in particular $f_0=\mathrm{id}$ differs from the limiting
constant map to $p$). Hence no single sequence $\{g_m\}$ can approximate every $f_t$ simultaneously,
so $\mathcal F$ is not Equi–Baire one on $K$.
\end{proof}

\begin{lemma}
\label{lem9}
Let $\{f_t\}_{t\ge0}=\{\exp(tA)\}_{t\geq0}$ be a one-parameter subgroup of $\mathrm{SL}(2,\mathbb C)$.
If the subgroup $\{\exp(tA)\}$ is relatively compact in $\mathrm{SL}(2,\mathbb C)$,
then for every compact $K\subset\widehat{\mathbb C}$ the family $\mathcal F=\{f_t:t\ge0\}$ is Equi–Baire one on $K$.
\end{lemma}

\begin{proof}
Assume $\{\exp(tA)\}$ is relatively compact and write $G=\overline{\{\exp(tA):t\ge0\}}$, the closure in $\mathrm{SL}(2,\mathbb C)$.
Then $G$ is a compact (abelian) subgroup of $\mathrm{SL}(2,\mathbb C)$. Consider the evaluation map
$$
\Phi:G\times K\longrightarrow\widehat{\mathbb C},\qquad \Phi(g,z)=g\cdot z,$$
which is continuous because the group action of $\mathrm{SL}(2,\mathbb C)$ on $\widehat{\mathbb C}$ is continuous.
For each $g\in G$ the restriction $g|_K$ is a continuous map $K\to\widehat{\mathbb C}$, hence
the set
$$
\mathcal G_K:=\{g|_K:g\in G\}\subset C(K,\widehat{\mathbb C})$$
is the continuous image of the compact set $G$ and therefore compact in the uniform topology on $C(K,\widehat{\mathbb {C}})$.

Compact metric spaces are separable, so there exists a sequence $\{h_n\}_{n\geq 1}\subset\mathcal G_K$ dense in $\mathcal {G}_K$
with respect to the uniform norm. In particular, for every $f\in\mathcal {F}$ (note $\mathcal F\subset\mathcal {G}_K$)
we have $h_n\to f$ uniformly on $K$, hence $h_n(x)\to f(x)$ for every $x\in K$. Since each $h_n$ is continuous,
the sequence $\{h_n\}$ provides a single sequence of continuous functions on $K$ that converges pointwise to every
$f\in\mathcal {F}$. Therefore $\mathcal {F}$ is Equi–Baire one on $K$.
\end{proof}
 
\section{Results regarding Equi-Baire One Families of Möbius Transformations and One–Parameter Subgroups of $\mathrm{PSL}(2,\mathbb{C})$ }
\label{sec4}
\subsection{Proof of Theorem \ref{thm main1}}

Every Möbius map $f(z)=\dfrac{az+b}{cz+d}$ with $ad-bc=1$ has fixed points given by 
$cz^2 + (d-a)z - b = 0$.  
In the loxodromic case, this quadratic has two distinct roots $p,q$, and the eigenvalues 
$\lambda_1,\lambda_2$ of its matrix satisfy $\lambda_1\lambda_2 = 1$ with 
$|\lambda_1| < 1 < |\lambda_2|$.  
A Möbius transformation $h$ sending $p \mapsto 0$ and $q \mapsto \infty$, for example 
$h(z)=\dfrac{z-p}{z-q}$, gives
$$
g = h \circ f \circ h^{-1},
$$
which fixes $0$ and $\infty$, hence $g(z) = \lambda z$ for some $\lambda \in \mathbb{C}^*$.
Since $f$ is loxodromic, $0 < |\lambda| < 1$
therefore $f = h^{-1} \circ g \circ h$ with the stated properties. 

\medskip
 
From lemma \ref{lem2} let $x$ lie in the attracting basin of $p$.  
Then there exists a compact neighborhood $K$ of $x$ such that 
$$
f^n|_K = h^{-1} \circ g^n \circ h|_K \longrightarrow p
$$
uniformly on $K$.

Choose $U$ around $x$ where $h$ is biholomorphic onto $V=h(U)$.  
Let $K=\overline{U}$ and $\overline{V}=h(K)$.  
Since $g(z)=\lambda z$ with $0<|\lambda|<1$, $g^n(z)=\lambda^n z \to 0$ uniformly on $\overline{V}$.  
By continuity of $h^{-1}$ on $\overline{V}$,
$$
f^n(z)=h^{-1}(g^n(h(z))) \to h^{-1}(0)=p
$$
uniformly on $K$. 

\medskip

From Lemma \ref{lem3} each $f^n|_K$ is continuous (hence Baire one), and the sequence $\{f^n|_K\}$ converges uniformly to $p$ on $K$.  
By  \cite[Th.~3.9]{ALI}, the family $\{f^n|_K\}$ is Equi–Baire one on $K$.

Uniform convergence on the compact metric space $K$ follows from Lemma \ref{lem2}.  
Each $f^n$ is a Möbius map, so continuous and Baire one.  
By  \cite[Th.~3.9]{ALI}, a uniformly convergent sequence of Baire one functions on a compact metric space is Equi–Baire one.  
Thus $\{f^n|_K\}$ is Equi–Baire one, giving for every $\varepsilon>0$ a $\delta>0$ such that 
$d(f^n(y), f^n(x))<\varepsilon$ for all $n$ whenever $d(y,x)<\delta$. 
\medskip

From Lemma \ref{lem4} define
$$
S(x,r) = \sup_{n\ge0}\, \sup_{y \in B_r(x)} d(f^n(y), f^n(x)).
$$
Then $S(x,r)$ is finite, $S(x,r) \to 0$ as $r \to 0$, and $r \mapsto S(x,r)$ is upper semicontinuous.

Compactness of $K$ and continuity of each $f^n$ imply finiteness.  
Let $\varepsilon>0$.  
From Lemma \ref{lem2}, $f^n \to p$ uniformly, so for large $n\geq N$,
$d(f^n(y),p)<\varepsilon/3$ for $y \in B_r(x)$.  
For $n<N$, continuity at $x$ gives $\delta>0$ with 
$d(f^n(y),f^n(x))<\varepsilon/3$ for $y \in B_\delta(x)$.  
Hence for all $n\geq 0$,
$d(f^n(y),f^n(x))<\varepsilon$ if $y\in B_\delta(x)$, proving $S(x,r)\to0$.  
Upper semicontinuity follows since enlarging $r$ only increases the domain of the supremum. 

\medskip 
From \ref{lem5} there exists $C>0$ such that for all $n\geq 0$ and $y$ near $x$,
$$
|f^n(y)-f^n(x)| \le C\,|\lambda|^n\,|h(y)-h(x)|.
$$
Hence $S(x,r) \le C'r$ and $S(x,r)\to0$ as $r\to0$.

Since $f^n=h^{-1}\circ g^n\circ h$ with $g(z)=\lambda z$,
$$
|f^n(y)-f^n(x)| = |h^{-1}(\lambda^n h(y)) - h^{-1}(\lambda^n h(x))|
\le C\,|\lambda|^n\,|h(y)-h(x)|
$$
for some $C>0$ (Lipschitz constant of $h^{-1}$ on $V$).  
Because $h$ is locally Lipschitz, $|h(y)-h(x)| \le L |y-x|$, giving 
$|f^n(y)-f^n(x)| \le C'|\lambda|^n |y-x|$.  
Taking the supremum over $y \in B_r(x)$ yields $S(x,r) \le C'r$. 

\medskip

\subsection{ \textit{Proof of Theorem \ref{thm main2}}}

We divide the proof into two parts.

\medskip
\noindent\textbf{Sufficiency:}
Assume that the one–parameter subgroup 
$G=\{\exp(tA): t\geq 0\}$ has relatively compact closure in 
$\mathrm{SL}(2,\mathbb{C})$. By Lemma~\ref{lem7}, each $f_t=\exp(tA)$ is a 
homeomorphism of $\widehat{\mathbb{C}}$, and hence each restriction 
$f_t|_{K}$ to a compact set $K\subset\widehat{\mathbb{C}}$ is continuous 
(and therefore Baire--one). By Lemma~\ref{lem9}, whenever $\{\exp(tA)\}$ is 
relatively compact, the family $F|_{K}=\{f_t|_{K} : t\geq 0\}$ is 
Equi--Baire one on $K$.

\medskip
\noindent\textbf{Necessity:}
Suppose next that the subgroup 
$$H=\{\exp(tA): t\in\mathbb{R}\}$$
is not relatively compact in $\mathrm{SL}(2,\mathbb{C})$.  
By the classification of one--parameter subgroups given in Lemma \ref{lem6}, $H$ must be hyperbolic, parabolic, or loxodromic.  In each of these three cases, the action on $\widehat{\mathbb{C}}$ admits an attracting (or parabolic limit) point $p\in\widehat{\mathbb{C}}$.  
Consequently, there exist a nonempty open set $U\subset\widehat{\mathbb{C}}$ and a sequence $t_n\to\infty$ such that 
$$f_{t_n}(z)\rightarrow p \qquad \text{for every } z\in U .$$

Let $K\subset\widehat{\mathbb{C}}$ be any compact set with $U\subset K$.  
By Lemma \ref{lem8}, if there exists a nonempty open set $U\subset K$ and a sequence 
$t_n\to\infty$ with $f_{t_n}(z)\to p$ for all $z\in U$, then the family 
$\{f_t : t\geq 0\}$ cannot be equi--Baire one on $K$.  
Therefore $F$ fails to be equi--Baire one on every compact neighborhood of $U$.

\medskip
\noindent
Combining the two directions, we conclude that the family 
$F=\{f_t : t\ge 0\}$ is equi--Baire one on every compact subset of 
$\widehat{\mathbb{C}}$ if and only if the subgroup $\{\exp(tA)\}$ is relatively 
compact in $\mathrm{SL}(2,\mathbb{C})$, which holds precisely when it is conjugate into $\mathrm{SU}(2)$.

\medskip
\noindent
\section{Conclusion}
Lemmas \ref{lem1}-\ref{lem5} along with Theorem \ref{thm main1} shows that the iterates $\{f^n\}$ converge uniformly to $p$ 
on a compact neighborhood $K$ of $x$, that the family is Equi–Baire one there 
(by \cite[Th.~3.9]{ALI}), and that the function $S(x,r)$ serves as an explicit 
$\delta$–function.  
Hence the family $F=\{f^n:n\ge0\}$ is orbitally Equi–Baire one at $x$.

Combining the sufficiency and necessity arguments of Theorem \ref{thm main2}, we conclude that the
Equi--Baire one behaviour of the one--parameter family $F=\{f_t : t \geq 0\}$ is
completely determined by the dynamical type of the subgroup
$\{\exp(tA) : t \in \mathbb{R}\} \subset \mathrm{SL}(2,\mathbb{C})$. If this subgroup is relatively compact---equivalently, conjugate into $\mathrm{SU}(2)$---then its action on $\widehat{\mathbb{C}}$ is uniformly regular on compact sets, and by Lemma~ \ref{lem9} there exists a single sequence of continuous functions that converges pointwise to every $f_t$ on any compact $K \subset \widehat{\mathbb{C}}$, proving that $F$ is Equi--Baire one on $K$.

Conversely, when the subgroup is hyperbolic, parabolic, or loxodromic, the action
admits an attracting (or parabolic limit) point, leading to collapsing behaviour on nonempty open sets. By Lemma~ \ref{lem8}, such behaviour prevents any single sequence of continuous functions from simultaneously approximating all maps in $F$, so the family fails to be Equi--Baire one on every compact neighborhood in these cases.

Therefore, $F=\{f_t : t \ge 0\}$ is Equi--Baire one on every compact subset of
$\widehat{\mathbb{C}}$ if and only if the subgroup $\{\exp(tA)\}$ is relatively compact in $\mathrm{SL}(2,\mathbb{C})$.

  \section*{Declaration of competent interest }
    The authors declare no competing interests.
    \section*{Data availability}
    No data was used for the research described in the article.
    \section*{Acknowledgments}
		Dutta acknowledges the Mizoram University for the Research and Promotion Grant F.No.A.1-1/MZU(Acad)/14/25-26. Vanlalruatkimi and Ramdikpuia also acknowledge Mizoram University for the Research Fellowship, file No.2-5/MZU(Acad)24/PF.

\end{document}